%% file: S26052.tex
\documentclass{article}
\usepackage{amsmath,amssymb,amsthm,stmaryrd,makeidx}
\usepackage[all]{xy}

\input{defn}

\input xypic

\newtheorem{proposition}{Proposition}

\newtheorem{lemma}{Lemma}
\newtheorem{corollary}{Corollary}
\newtheorem{definition}{Definition}

\newtheorem{remark}{Remark}

\newtheorem{example}{Example}

\makeindex

\begin{document}
\author{Arvid Siqveland}
\title{Sites and Grothendieck Topologies\\Sites and Sheaves}

\maketitle

\begin{abstract} We give the minimum of category theory necessary for understanding the definition and applications of Grothendieck topos. We state the basic properties of sites and sheaves, and give applications to the theory of moduli. We prove that for categories $\mathbf C$ with  explicit stated properties, we can construct a category of schemes of objects in $\mathbf C$ which is a proper category rather than a $2$-category as in the case of stacks. 
\end{abstract}

\section{Introduction}
This is the lecture notes for my contribution to the summer school on $K$-theory at the University of Stavanger named "Towards Quadratic Intersection Theory", to be held 17. - 18.08.2026. My talk will be the content of Section \ref{lecturechapter}, the preliminaries is given in  the earlier sections. The sections following Section \ref{lecturechapter} proves that we can construct moduli schemes of objects in abelian categories with Cartesian products.

\section{Category Theory and the Yoneda Lemma}

The core philosophy is that in the search for a moduli object (parametrizing mathematical objects $F(M)$), the ultimate goal is to find a category $\cat U$ such that it contains a unique object $M\in\ob\cat U$ which is a moduli (space) for $F(B).$

\subsection{Categories}

\begin{definition} A category $\cat C$ is the pair $(\ob\cat C,\{\mor_{\cat C}(C_1,C_2)\}_{C_1,C_2\in\ob\cat C})$ where $\ob \cat C$ is a collection of objects, and for each pair of objects $C_1,C_2\in\ob\cat C,$ $\mor_{\cat C}(C_1,C_2)$ is a collection of objects called morphisms from $C_1$ to $C_2,$ and denoted $\phi:C_1\rightarrow C_2,$ with the following two properties:
\begin{itemize}
\item[1.] For each triple $C_1,C_2,C_3,$ there exists an associative operation $$\circ:\mor_{\cat C}(C_1,C_2)\times\mor_{\cat C}(C_2,C_3)\rightarrow\mor_{\cat C}(C_1,C_3).$$ 
\item[2.] For each $C\in\ob\cat C$ there is an element $\id_C\in\mor_{\cat C}(C,C)$ such that for $\phi:C\rightarrow D$ we have $\id_D\circ\phi=\phi=\phi\circ\ \id_C.$ 
\end{itemize}
For every category $\cat C$ the opposite category $\cat C^\text{op}$ is the category with the same objects as $\cat C$ and with $\mor_{\cat C^\text{op}}(C_1,C_2)=\mor_{\cat C}(C_2,C_1).$
\end{definition}

Notice that we do not give any internal information about the objects in a category. The objects are needed only as representatives for the start (tail) and goal (head) of the morphisms (arrows). We have the equivalent definition of a category as a collection of arrows with head and tail that can be composed associatively when the head of the first coincides with the tail of the second. Comparing with mathematical logic and set theory, this is to say that some properties follows from propositional logic (the Zermelo-Fraenkel axioms that we assume), while De Morgans's laws follows from the internal structure of the category of sets, $\cat{Sets}.$

\begin{definition} Two objects $C_1,C_2$ in a category $\cat C$ are called isomorphic, written $C_1\simeq C_2,$ if there are morphisms $\psi:C_1\rightarrow C_2$ and $\psi^{-1}:C_2\rightarrow C_1$ such that $\psi\circ\psi^{-1}=\id_{C_2}$ and $\psi^{-1}\circ\psi=\id_{C_1}.$
\end{definition}

\begin{definition} Let $\cat C,\cat D$ be two categories. A covariant functor $ F:\cat C\rightarrow\cat D$ between  the categories $\cat C$ and $\cat D$ consists of:
\begin{itemize} 
\item[1.] A map $F:\ob\cat C\rightarrow\ob\cat D.$
\item[2.] For each pair $C_1,C_2\in\ob\cat C$ a map $$F:\mor_{\cat C}(C_1,C_2)\rightarrow\mor_{\cat D}(F(C_1),F(C_2)).$$
\end{itemize}
Such that for each triple $C_1,C_2,C_3\in\ob\cat C$ and each $$(\phi,\psi)\in\mor_{\cat C}(C_1,C_2)\times\mor_{\cat C}(C_2,C_3),$$ 
$F(\phi)\circ F(\psi)=F(\phi\circ\psi).$ We also demand that $F(\id_C)=\id_{F(C)}.$
A functor $ F:\cat C\rightarrow\cat D^{\text{op}}$ is called a contravariant functor. 
\end{definition} 

A functor is then contravariant if it turns the arrows in their opposite directions. Given a category $\cat C,$ we define the identity functor $\id_{\cat C}:\cat C\rightarrow\cat C$ which is the identity both on objects and morphisms.

\begin{definition} Let $F:\cat C\rightarrow\cat D$ be a functor. If for all pairs $C_1,C_2\in\cat C$ the map $F:\mor_{\cat C}(C_1,C_2)\rightarrow\mor_{\cat D}(F(C_1),F(C_2))$ is surjective, it is called full. If it is injective, it is called faithful. If it is both surjective and injective, it is called fully faithful.
\end{definition}

\begin{definition} Given two functors $\operatorname{F},\operatorname{G}:\cat C\rightarrow\cat D.$ A natural transformation $T:F\rightarrow G$ consists of, for each $C\in\ob\cat C,$ a morphism $T(C):\operatorname{F}(C)\rightarrow\operatorname{G}(C)$ in $\cat D,$ such that if $\phi:C\rightarrow D$ is a morphism in $\cat C$ the following diagram commutes.
$$\xymatrix{F(C)\ar[d]_{F(\phi)}\ar[r]^{T(C)}&G(C)\ar[d]^{G(\phi)}\\F(D)\ar[r]_{T(D)}&G(D).}$$ Two functors $\operatorname{F},\operatorname{G}:\cat C\rightarrow\cat D$ are isomorphic, written $\operatorname{F}\simeq\operatorname{G},$ if there is a natural transformation $T:F\rightarrow G$ such that for each $C\in\ob\cat C,\ T(C):F(C)\rightarrow G(C)$ is an isomorphism. 
\end{definition}

\begin{definition} A functor $ F:\cat C\rightarrow\cat D$ is called an equivalence of categories if there exists a functor $\operatorname G:\cat D\rightarrow\cat C$ such that $F\circ G\simeq\id_{\cat D}$ and $G\circ F\simeq\id_{\cat C}.$
\end{definition}

Here are some examples of categories and functors.

\begin{example}
\begin{itemize}
\item[i)] The category $\cat{Sets}$ of sets, where the morphisms are just maps of sets.
\item[ii)] The category $\cat{Top}$ where the objects are topological spaces, and the morphisms are continuous maps. 
\item[iii)] $\cat{Grp}$ denotes the category where the objects are groups, and the morphisms are group homomorphisms.
\item[iv)] The category $\cat{Vec}_{\mathbb Q}$ of rational vector spaces with linear transformations as morphisms.
\item[v)] In the four last examples there is a functor $S:\cat C\rightarrow\cat{Sets}$ sending an object to its underlying set.
\end{itemize}
\end{example}

\begin{definition} A category $\cat C$ is locally small if all collections $\mor_{\cat C}(C_1,C_2)$ are sets for all $C_1,C_2\in\cat C.$ If, in addition, the collection $\ob C$ is a set, the category is called small. A category which is not small, is called large.
\end{definition}

We can define the category $\cat{Cat}$ of small categories, where the objects are small categories and the morphisms are functors. Then we have that $\cat{Cat}$ is a small category as well.  However, the category of categories would fail to be a category, because a large category is not necessarily an element in the category of categories by Russel's paradox.

All our examples up to now are small, but they are even easier.

\begin{definition} A category $\cat C$ is concrete if there exists a faithful functor $$S:\cat C\rightarrow\cat{Sets}.$$
\end{definition}

This is another way to say that a category $\cat C$ is concrete if its objects are sets, and its morphisms are contained in the set of maps of sets. All our examples up to now are concrete categories.

\subsection{Moduli objects}
In this section we will enter the structure of our objects in a small category. That is, we will assume that it is possible to talk about elements in objects. Thus we are stepping out of abstract category theory, and work more concretely. We could however replaced the concept of a nonempty set $X$ with the existence of a morphism from an initial object $P$ to  $X.$  
\begin{remark} To pin the importance: Some theorems on objects in a category is in need of the internal structure of the objects. In this section we enter the internal structure.
\end{remark}
Let $\cat C$ be a small category and $ F:\cat C\rightarrow\cat{Sets}$  a covariant functor. For each $C\in\ob\cat C,$  we have the covariant functor $\mor_{\cat{C}}(C,-):\cat C\rightarrow\cat{Sets}$ and we define the map of sets $$y_C:F(C)\rightarrow\mor_{\cat{Funct}}(\mor_{\cat C}(C,-),F),$$ from the set $F(C)$ to the set of natural transformations between the two functors, by the following: $y_C(\xi):\mor_{\cat C}(C,-)\rightarrow F$ is the natural transformation of functors given by $$y_C(\xi)(D)(\phi)=F(\phi)(\xi)\in F(D),\ \xi\in F(C).$$ The commutativity on the morphisms in $\cat C$ follows from the diagram $$\xymatrix{\mor_{\cat C}(C,D_1)\ar[r]^-{y_C(D_1)}\ar[d]_{\phi_\ast}&F(D_1)\ar[d]^{F(\phi)}\\\mor_{\cat C}(C,D_2)\ar[r]_-{y_C(D_2)}&F(D_2)}$$ where $\phi:D_1\rightarrow D_2$ is a morphism in $\cat C$ and $\phi_\ast=\mor_{\cat C}(D_1,\phi).$

\begin{proposition}{(Yoneda's Lemma)}\label{Yoneda} For every small category $\cat C$ and any functor $F:\cat C\rightarrow\sets,$ the Yoneda map 
$$y_C:F(C)\rightarrow\mor_{\cat{Cat}}(\mor_{\cat C}(C,-),F)$$ is a bijection for every $C\in\cat C.$
\end{proposition}

\begin{proof} We construct an inverse natural transformation: Assume that $$\operatorname{Y}_C: F\rightarrow\mor(C,-)$$ is a natural transformation of functors. Then $Y_C(C)(\id)=\xi\in F(C)$ and we see that $y_C(\xi)=\operatorname Y_C.$
\end{proof}

\begin{definition} Let $\cat C$ be a small  category and $ F:\cat C\rightarrow\cat{Sets}$  a covariant functor. An object $M\in\ob\cat C$ together with an element $\xi\in F(M)$ represents $ F,$ and $ F$ is representable by $(M,\xi),$ if  the Yoneda map $y_M(\xi):\mor_{\cat C}(C,-)\rightarrow F$ is an isomorphism of functors. 
\end{definition}

\begin{example} Let $\cat C$ be a small category and fix an object $P\in\ob\cat C.$ Then $\mor_{\cat C}(P,-)$ is a covariant functor, and for every $C\in\ob\cat C$ we have a bijection $$\mor_{\cat C}(P,C)\rightarrow\mor_{\cat{Funct}}(\mor_{\cat C}(C,-),\operatorname \mor_{\cat C}(P,-))$$ by Yoneda's Lemma \ref{Yoneda}. Letting $\xi\in\mor_{\cat C}(P,C),$ the Yoneda morphism $y_C$ is given for each $D$ as $$y_C(\xi)(D):\mor_{\cat C}(C,D)\rightarrow\mor_{\cat C}(P,D)$$ by sending $\phi\in\mor_{\cat C}(C,D)$ to $\phi\circ\xi\in\mor_{\cat C}(P,D).$
\end{example}

Our next definition is in need of an explanation. We consider a small category $\cat C,$ and let $F:\cat C\rightarrow\cat{Sets}$ be a covariant functor represented by $(M,\xi),$ i.e., for each object $C$ of $\cat C,$ we have a natural bijection
$$\mor_{\cat C}(M,C)\tilde\longrightarrow F(C).$$ Each element $f\in F(C)$ corresponds to a morphism $\phi_f:M\rightarrow C,$ which we can call a $C$-point of $M,$ and we can say that $M$ parametrizes the elements in $F(C).$ If $D$ is another object containing a $C$-point $\phi:D\rightarrow C,$ we have the induced diagram $$\xymatrix{\mor_{\cat C}(M,D)\ar[r]\ar[d]&F(D)\ar[d]\\\mor_{\cat C}(M,C)\ar[r]&F(C),}$$ and that is to say that there exists a morphism $\psi:M\rightarrow D$ sending the $C$-point $\phi$ in $D$ to the $C$-point $\phi\circ\psi$ in $M.$

\begin{definition} Let $\cat C$ be a small category,  and let $ F:\cat C\rightarrow\cat{Sets}$ be a covariant functor. If $(M,\xi)$ represents $ F$ we say that $(M,\xi)$ is a moduli object for the set $F(P)$ for all $P\in\ob\cat C.$ 
\end{definition}

\begin{example}
Let $T$ be an equilateral triangle in $\mathbb Q^2$ and let $S_3$ be the set of linear isometries $S:\mathbb Q^2\rightarrow\mathbb Q^2$ which satisfies $S(T)=T.$ We consider the free group $\langle x\rangle$ in one indeterminate and look for a category $\cat C$ which contains $\langle x\rangle$ as an object, and such that there is an object $M\in\ob\cat C$ which is a moduli object for $S_3.$ This says $\mor_{\cat C}(\langle x\rangle,M)=S_3.$ If we require that $\cat C$ is concrete, and assume that for all $\phi:\langle x\rangle\rightarrow S_3$ we should have $\phi(x^n)=\phi(x)^n$ for all $n,$ the only possibility is to let $\cat C=\cat{Gr},$ the category of groups, and $M=\operatorname S_3,$ the permutation group on three elements. 
\end{example}

\begin{example} We define the category of vector spaces $\cat{Vec}_{\mathbb Q}$ over $\mathbb Q$ as the concrete category containing $\mathbb Q$ as an object, and which contains an object $V$ representing $\mor_{\cat{Vec}_{\mathbb Q}}(-,\mathbb Q).$ Here $\mor_{\cat{Vec}_{\mathbb Q}}(\mathbb Q,\mathbb Q)\simeq\mathbb Q$ has to be the set of linear homomorphisms, and  this determines the category.
\end{example}

\begin{example} The category $\cat{Top}$ of topological spaces is the category containing a moduli for the metric space $\mathbb Q,$ determined by the demand on continuous functions $f:\mathbb Q\rightarrow\mathbb Q.$
\end{example}

\subsection{Universal Properties}

When we work in a fixed small category $\cat C$, we often define certain objects by universal properties. This usually means  that a couple $(M,\xi)$ represents a functor $F:\cat C\rightarrow\cat{Sets}.$ 

\begin{lemma} Let $F:\cat C\rightarrow\cat{Sets}$ be a functor from a small category into $\cat{Sets}.$ If $(M,\xi)$ and $(N,\phi)$ are two representing objects, there exist an isomorphism $\iota:M\rightarrow N$ such that $F(\iota)(\xi)=\psi.$
\end{lemma}

\begin{proof}
We have the two bijections $$y_M(\xi)(N):\mor_{\cat C}(M,N)\rightarrow F(N)\simeq\mor_{\cat C}(N,N)$$ and
$$y_N(\phi)(M):\mor_{\cat C}(N,M)\rightarrow F(M)\simeq\mor_{\cat C}(M,M).$$ The first proposition says that there is a morphism $\iota:M\rightarrow N$ such that $F(\iota)(\xi)=\psi,$ the seoond says that there exists an inverse, proving the statement.
\end{proof}

\begin{example}{(Categorical product).}
Let $C_1$ and $C_2$ be objects in a category $\cat C.$ The categorical product of $C_1$ and $C_2$ is an object $C_1\prod C_2$ representing the functor $F(C_1,C_2):\cat C\rightarrow\cat{Sets}$ given by $$F(C_1,C_2)(P)=\mor_{\cat C}(P,C_1)\times\mor_{\cat C}(P,C_2).$$
The more common way to say this is that the product $C_1\prod C_2$ is an object determined by the universal property that there there exists morphisms $$\pi_1:C_1\prod C_2\rightarrow C_1,\ \pi_2:C_1\prod C_2\rightarrow C_2,$$ and such that if $Y$ is another object with morphisms $$p_1:Y\rightarrow C_1,\ p_2:Y\rightarrow C_2,$$ then there exists a unique morphism $Y\rightarrow C_1\prod C_2$ making the following diagram commutative:
$$\xymatrix{&Y\ar[dl]_{p_1}\ar[dr]^{p_2}\ar[d]&\\C_1&\ar[l]^{\pi_1}C_1\prod C_2\ar[r]_{\pi_2}&C_2.}$$
\end{example}

\begin{example}{(Product and Coproduct in $\cat{Sets}$).}
Let $A$ and $B$ be sets. Then $A\prod B$ fulfil the diagram $$\xymatrix{X\ar@{-->}[r]^{\exists !}\ar[d]\ar[dr]&A\prod B\ar[d]\ar[dl]\\A&B.}$$ Then we see that the ordinary Cartesian product will do the job. 
 The coproduct of $A$ and $B$ fulfils the diagram
 
 $$\xymatrix{X&A\coprod B\ar@{-->}[l]^{\exists !}\\A\ar[ur]_{i_1}\ar[u]_{j_1}&B\ar[u]_{i_2}\ar[ul]_{j_2}.}$$
 
 We see that we have to choose $b\in B$ to  construct a map $i_1^b:A\rightarrow A\coprod B,\ i_1^b(a)=(a,b)$ and likewise for $i_2,$ which says that there is no unique map $A\coprod B\rightarrow X.$ To prove that a coproduct exists, put $$A\coprod B=(A\times\{0\})\cup(B\times\{1\}),$$ and the unique mapping $A\coprod B\rightarrow X$ must be $\phi((a,0))=j_1(a)$ and $\phi((b,1))=j_2(b)).$
\end{example}

\begin{example}{(Fibred product).} Let $\cat C$ be a category and $U\in\ob\cat C.$ For two morphisms $$X\overset f\rightarrow U,\ Y\overset g\rightarrow U,$$ their fibred product, if it exists, is an object $X\times_U Y$ with morphisms $\phi:X\times_U Y\rightarrow X$ and $\psi:X\times_U Y\rightarrow Y$ such that $f\circ\phi=g\circ\psi,$ and such that if $P$ is any other object with morphisms $\phi':P \rightarrow X$ and $\psi':P\rightarrow Y$ with $f\circ\phi'=g\circ\psi',$ then there exists a unique morphism $\xi:P\rightarrow X\times_U Y$ such that $\phi\circ\xi=\phi'$ and $\psi\circ\xi=\psi'.$

$$\xymatrix{P\ar@{-->}[rr]^\xi\ar[dr]^{\phi'}\ar[drrr]^{\psi'}&&X\times_U Y\ar[dl]^\phi\ar[dr]^\psi&\\&X\ar[dr]^f&&Y\ar[dl]^g\\&&U.&}$$

\end{example}
 
\begin{example}{(Projective (Inverse) and Inductive (Direct) Limits).}
A partially ordered set $(D,\leq)$ is directed if for all $a,b\in P$ there exists an element $c\in D$ such that both $a\leq c\land b\leq c.$  A directed partially ordered set $D$ then defines a category $\cat D$ where $\ob\cat D=D$ and $$\mor_{\cat D}(d_1,d_2)=\begin{cases}\{\leq\},\text{if }d_1\leq d_2,\\\emptyset\text{ otherwise.}\end{cases}$$ Let $\cat C$ be a small category. A projective system in $\cat C$ over $\cat D$  is a contravariant functor $F:\cat D\rightarrow\cat C.$ For an object $C\in\cat C$ we will call an element $\phi\in\prod_{\cat D}\mor_{\cat C}(C,F(D))$ commuting if for all $i\geq j\geq k$ in $\cat D$ we have that $\phi_k=\phi_i\circ F(i\geq j)\circ F(j\geq k).$
Let $F:\cat D\rightarrow\cat C$ be a projective system and for each $C\in\ob\cat C$ consider the set $$P(C)=\{\phi_C\in\underset{D}\prod\mor_{\cat C}(C,F(D))|\phi_C\text{ is commuting}\}.$$ This defines a contravariant functor $P:\cat C\rightarrow\cat{Sets},$ and a representing object $\underset{\underset{d\in D}\longleftarrow}\lim P(d)\in\ob\cat C$ is called the projective limit of $F.$ In the traditional language, this translates to: The projective limit of $F$ over $D$ consists of an object $\underset{\underset{d\in D}\longleftarrow}\lim P(d)\in\ob\cat C$ satisfying the following universal property: There exists a morphism $\phi_d:\underset{\underset{d\in D}\longleftarrow}\lim P(d)\rightarrow P(d)$  for each $d\in D$ commuting with the morphisms in $\cat D,$ such that if $(X,(\psi_d)_{d\in D})$ is another such object, there exists a unique morphism $\phi:X\rightarrow\underset{\underset{d\in D}\longleftarrow}\lim P(d)$ such that $\phi_d\circ\phi=\psi_d$ for each $d\in D.$

The inductive limit is defined by considering the covariant functor $$I(C)=\{\phi_C\in\underset{D}\prod\mor_{\cat C}(F(D),C)|\phi_C\text{ is commuting}\}$$ which is represented by the inductive limit of $F$ over $\cat D$ by $\underset{\underset{d\in D}\longrightarrow}\lim\ I(d).$ The traditional definition is to turn all the arrows in the definition of projective limits.
\end{example}

\begin{definition} Let $\cat C$ be a category. If all inductive limits exists, $\cat C$ is called complete. If all projective limits exists, it is called cocomplete. If all fibred products exists, it is called Cartesian closed.
\end{definition}

\subsection{Sheaves on Topological spaces}

The topology on a topological space $X$ defines a category $\cat{Top}(X)$ where the objects are the open sets, and the morphisms are the inclusions of sets. We will study functors $F:\cat{Top}(X)\rightarrow\cat C$ where $\cat C$ is of a particular kind.

A category $\cat C$ is called concrete if there exists a faithful functor $S:\cat C\rightarrow\cat{Sets}.$ We will only consider concrete categories $\cat C$  which contains at least one object $P$ such that $\mor_{\cat C}(P,C)=S(C).$ Examples of such categories are,

\begin{itemize}
\item[(1)] $\cat{Grps},$ the category of groups which contains the free group on one element $P=\langle x\rangle$ for which the group homomorphisms $P\rightarrow G$ is in one to one correspondence with the elements in the group $G,$ 
\item[(2)] $\cat{Rings},$ the category of rings with unit, which contains the free polynomial ring $P=\mathbb Z[x]$ in one variable over $\mathbb Z$ for which the ring homomorphisms $P\rightarrow R$ is in one to one correspondence with the elements in the ring $R.$
\end{itemize}

For such categories, a contravariant functor $F:\cat{Top}(X)\rightarrow\cat C$ is called a presheaf of $\cat C$-objects on $X.$  With the assumption on the category $\cat C$ we can formulate Definition II.(1.2), \cite{HH77}, in Hartshorne the following way

\begin{definition} A presheaf $\mathcal F$ of groups on $X$ is called a sheaf if the following two conditions hold:
\begin{itemize}
\item[(1)] If $U$ is an open set, if $\{V_i\}_{i\in I}$ is an open covering of $U,$ and if $s:P\rightarrow\mathcal F(U)$ satisfies $s|_{V_i}=0$ for all $i,$ then $s=0$;
\item[(2)] If $U$ is an open set, if $\{V_i\}_{i\in I}$ is an open covering of $U,$ and if we have $s_i:P\rightarrow\mathcal F(V_i)$ for each $i\in I,$ with the property that for each $i,j,\ s_i|_{V_i\cap V_j}=\ s_j|_{V_i\cap V_j},$ there is an $s:P\rightarrow \mathcal F(U)$ such that $s|_{V_i}=s_i$ for each $i\in I.$
\end{itemize}
\end{definition}

\begin{lemma}\label{asssheaf} Let $F$ be a presheaf on $X.$ Then there exists a sheaf $\mathcal F$ on $X$ with a morphism $\theta:F\rightarrow\mathcal F$ such that if $\mathcal G$ is another sheaf with a morphism $\phi:F\rightarrow\mathcal G,$ there is a unique morphism $\psi:\mathcal F\rightarrow\mathcal G$ such that $\phi=\psi\circ\theta.$
\end{lemma}

\begin{proof} We prove that the presheaf $\mathcal F$ defined by $\mathcal F(U)=\underset{\underset{V\subsetneq U}\longleftarrow}\lim\ F(V)$ is a sheaf.

\begin{itemize} \item[(1)] If $s:P\rightarrow\mathcal F(U)$ restricts to $s|_{V_i}=0$ for each $i\in I,$ then $s$ has to be equal to $0:P\rightarrow\mathcal F(U)$ by the universal property of projective limits.
\item[(2)] The condition says that for each subset $V\subsetneq U$ we have a morphism $s_V:P\rightarrow F(V)\rightarrow\mathcal F(V)$ commuting in all chains of open subsets. Then there exists a unique morphism $s:P\rightarrow\mathcal F(U)$ such that $s|_V=s_V.$ 
\end{itemize}
By the definition of projective limits, it follows that the sheaf $\mathcal F$ has the proposed condition.
\end{proof}

\begin{definition} Let $X$ be a topological space and $F$ a presheaf on $X.$ Then the sheaf $\mathcal F$ on $X$ in Lemma \ref{asssheaf}, is called the sheaf associated to the presheaf $F.$
\end{definition}

\begin{definition} Let $F$ be a presheaf on a topological space $X.$ Then for $x\in X$ we define the stalk of $F$ in $x$ as $F_x=\underset{\underset{x\in U}\longrightarrow}\lim\ F(U).$
\end{definition}

\begin{lemma}\label{stalklemma} Let $F$ be a presheaf on $X$ and let $\mathcal F$ be the associated sheaf. Then  for all $x\in X,$ $\mathcal F_x=F_x.$
\end{lemma}

\begin{proof} By the universal property there exists an isomorphism $F_x\simeq\mathcal F_x.$
\end{proof}

In the preprint \cite{S250} we give a categorical construction of ordinary schemes based on the theory of sheaves on topological spaces as considered in this subsection.

\section{Sites and Sheaves}\label{lecturechapter}

\subsection{Grothendieck Topology}
We follow [V07] verbatim (more or less).
\begin{definition} Let $\cat C$ be a category. A Grothendieck topology on $\cat C$ is the assignment to each object $U$ of $\cat C$ of a collection of sets of arrows $\{U_i\rightarrow U\},$ called coverings of $U,$ so that the following conditions are satisfied.
\begin{itemize}
\item[(i)] If $V\rightarrow U$ is an isomorphism, then the set $\{V\rightarrow U\}$ is a covering.
\item[(ii)] If $\{U_i\rightarrow U\}$ is a covering and $V\rightarrow U$ is any arrow, then the fibered products $\{U_i\times_U V\}$ exists, and the collection of projections $\{U_i\times_U V\rightarrow V\}$ is a covering.
\item[(iii)] If $\{U_i\rightarrow U\}$ is a covering, and for each index $i$ we have a covering  $\{V_{ij}\rightarrow U_i\}$ (here $j$ varies on a set depending on $i$), the collection of composites $\{V_{ij}\rightarrow U_i\rightarrow U\}$ is a covering of $U.$
\end{itemize}
A category with a Grothendieck topology is called a site.
\end{definition}

\subsection{Sheaves on Sites}

\begin{definition}\label{Sheafdef} Let $\cat C$ be a site, $F:\cat C^\text{op}\rightarrow\cat{Sets}$ a functor.
\begin{itemize}
\item[(i)] $F$ is separated if, given a covering $\{U_i\rightarrow U\}$ in $\cat C$ and two sections $a$ and $b$ in $F(U)$ whose pullbacks to each $F(U_i)$ coincide, it follows that $a=b.$
\item[(ii)] $F$ is a sheaf if the following condition is satisfied. Suppose that we are given a covering $\{U_i\rightarrow U\}$ in $\cat C,$ and a set of elements $a_i\in F(U_i).$ Denote by $\operatorname{pr}_1:U_i\times_U U_j\rightarrow U_i$ and $\operatorname{pr}_2:U_i\times_U U_j\rightarrow U_j$ the first and second projections respectively, and assume that $\operatorname{pr}_1^\ast(a_i)=\operatorname{pr}_2^\ast(a_j)\in F(U_i\times_U U_j)$ for all $i$ and $j.$ Then there is a unique section $a\in F(U)$ whose pullback to $F(U_i)$ is $a_i$ for all $i.$ If $F$ and $G$ are sheaves on a site $\cat C,$ a morphism of sheaves is simply a natural transformation of functors.
\end{itemize}
\end{definition}

We can strain an object through a sieve. Then it is properly strained, and straining it again will not effect the strained object.
\begin{definition} Let $U$ be an object of a category $\cat C.$ A sieve on $U$ is a subfunctor of $\mor_{\cat C}(-,U):\cat C^\text{op}\rightarrow\sets.$
\end{definition}

A presheaf $F$ on a site can be strained through a sieve to become a sheaf. Lemma \ref{Yoneda} (Yoneda's Lemma) says that $F(U)\overset\sim\rightarrow\mor(\mor(U,-), F).$ This defines a  presheaf $\mathcal F$ of a sieve, which is the sheaf $\mathcal F$ associated to the presheaf $F,$ and a presheaf is a sheaf if and only if $\mathcal F\simeq F.$ In the search for representing objects, or moduli, one can choose a sieve, which Vistoli call a cleavage, which is then a $2$ functor, giving rise to the $2$-category of stacks. 

We will use an equivalent and more general exposition where we apply our earlier construction by projective limits replacing the sieves.

Let $\cat C$ be a site. Then a sequence of morphisms $$V=\{V_i\rightarrow V_{i-1}\}_{i=n}^1=V_n\rightarrow V_{n-1}\rightarrow\cdots\rightarrow V_1\rightarrow V_0$$ is called a sequence of opens  if $V_i\rightarrow V_{i-1}$ is in the cover of $V_{i-1},\ n\leq i\leq 1.$ It is called a sequence of opens in $U$ if $V_0=U$; and when none of the composites $V_i\rightarrow V_0,\ i>0,$ is not an isomorphism, it is called proper.

\begin{definition}
Let $\cat C$ be a site, and let $F:\cat C^\text{op}\rightarrow\cat D$ be a functor (presheaf). We define the sheaf $\mathcal F$ of $\cat D$  associated to $F,$ by for each object $U$ in $\cat C,$ letting $\mathcal F(U)=\underset{\underset{V\rightarrow U}\leftarrow}\lim F(V),$ where the limit is taken over all proper sequences of opens in $U.$
\end{definition}

\begin{corollary} A functor $F:\cat C\rightarrow\cat D$ is a sheaf if and only if $F\simeq\mathcal F.$
\end{corollary}

We have to prove that when $\cat D=\cat{Sets}$ then our definition agrees with Definition \ref{Sheafdef}.
\begin{proposition} When $\cat C$ is a site and $F:\cat C^\text{op}\rightarrow\cat{Sets}$ is a contravariant functor, then $\mathcal F$ is a sheaf of Sets.
\end{proposition}

\begin{proof} Suppose we are given a covering $\{U_i\rightarrow U\}$ in $\cat C,$ and a set of elements $a_i\in \mathcal F(U_i)$ such that $\operatorname{pr}_1^\ast(a_i)=\operatorname{pr}_2^\ast(a_j)=:a_{ij}\in F(U_i\times_U U_j).$ Define $$f_{ij}:\{x\}\rightarrow\mathcal F(U_i\times_U U_j)$$  by $f_{ij}(x)=a_{ij}.$ By the universal property of projective limits, there is a map $f:\{x\}\rightarrow \mathcal F(U),$ and $a:=f(x)\in\mathcal F(U)$ pulls back to $a_i$ for all $i.$
\end{proof}

\begin{example}\label{Affex} Let $A$ be a commutative ring, and let $U=\spec A$ be its set of prime ideals. We give $U$ the topology generated by the sub-basis $\{D(f)\}_{f\in A}$ where $D(f)=\{\mathfrak p\in X|f\notin\mathfrak p\}.$ A ring homomorphism $\phi:A\rightarrow B$ gives a continuous function  $\phi^\ast:\spec B\rightarrow\spec A$ so that we can let the objects of the category $\cat{Aff}$ of affine schemes be the  $U=\spec A$ for rings $A,$ and the morphisms from $\spec B\rightarrow\spec A$  the ring homomorphisms $\phi:A\rightarrow B.$ When $\im\phi^\ast\subseteq\spec B$ is open, we call $\phi^\ast$ an open embedding. We define a Grothendieck topology on $\cat{Aff}$ by letting a covering $\{U_i\rightarrow U\}$ be a collection of open embeddings covering $U.$ This is the Zariski topology.  We have a functor $\mathcal O:\cat{Aff}^\text{op}\rightarrow\cat{Rings}$ given by $\mathcal O(\spec A)=A$ which is a sheaf by its definition as inverse to the $\spec$-functor.
\end{example}

\begin{example} We make Example \ref{Affex} global: We define a scheme of rings as  a topological space X which has an open affine cover $X=\underset{i\in I}\cup U_i$ such that for each $i,$ $U_i\simeq\spec A_i$ (homeomorphic) for a ring $A_i.$ A morphism of schemes $f:X\rightarrow Y$ is is a continuous map such that whenever $U=\spec B\subseteq X,$  $V=\spec A\subseteq Y$ and $U\subseteq f^{-1}(V),$  then $f|_U=\phi^\ast$ for a ring homomorphism $\phi:A\rightarrow B.$ We define a Grothendieck topology on the category of rings $\cat{Sch}.$ A covering $\{U_i\rightarrow U\}$ is a collection of open embeddings covering $U.$ This is the global Zariski topology. On this topology we have a natural definition of a sheaf $\mathcal O:\cat{Sch}^\op\rightarrow\cat{Rings}$ given by $\mathcal O(U)=\underset{\underset{V\rightarrow U}\leftarrow}\lim \mathcal O(V),$ where the limit is taken over all proper sequences of opens in $U.$
When $X$ is any scheme, the standard notation is $\mathcal O|_{\cat{Top}X}=\mathcal O_X.$
\end{example}

Let $\cat C$ be a site with coverings $U=\{U_i\rightarrow U\}_{i\in I}.$ If all ring homomorphisms in $U$ are finitely presented and locally:

\begin{itemize} 
\item[i)] Étale: The global étale topology.
\item[ii)] Flat: The fppf topology (fidèlement plat et de présentation finie).
\end{itemize} 

To avoid  finite presentation, when all homomorphisms in the covering are faithfully flat and quasi-compact, we get the fpqc topology ( fidèlement plat et quasi-compact).

\subsection{Grothendieck Topoi}

Let $\cat C$ be a site, and let $\cat{Shv}_{\cat C}$ denote the category of sheaves (of sets) on $\cat C.$ Because the objects in this category satisfy the universal properties of sheaves, every universal property in the category $\cat{Sets}$, can be lifted to the equivalent universal property for $\cat{Scv}_{\cat C}.$ Thus the logic in topological spaces, is equivalent to the logic in sheaves on  sites.

\begin{lemma} Let $\cat C$ be a site. Then the category $\cat{Shv}_{\cat C}$ is:
\begin{itemize}
\item[(i)] Complete and cocomplete.
\item[(ii)] Cartesian closed.
\end{itemize}
\end{lemma} 

\begin{definition} A Grothendieck Topos over sets, is a category $\cat{Shv}_{\cat C}$ for a site $\cat C.$
\end{definition}

A category $\cat C$ is called a Grothendieck category if it has the same properties as the category of sheaves of modules over a scheme.

We see that $\cat{Shv}_{\cat C}$ is a Grothendieck category if the objects are actually sheaves of abelian groups. Notice that these automatically are sheaves in the category of $\mathbb Z$-modules, so that global functions rings and stalks of function rings makes sense.

It is well known that we can get information about a particular variety, or manifold, by studying its sheaf of functions.
K-theory is about (among other things) studying a site $\cat C$ by the study of its Grothendieck topos. The underlying philosophy is that objects are classified by comparison, not as objects standing alone in a vacuum. So let $K$ be a set (or proper class) of objects that we want to study. We search for a small site $\cat K$ having $K$ as the set of objects. This is included in the theory of fibred categories and Descent.

\begin{definition} A stack is an object in a Grothendieck topos. 
\end{definition}

\section{Grothendieck Schemes of objects in Abelian Categories}\label{SOACsection}

\subsection{Base Points and Local Representability}
Let $\cat C$ be a locally small category, i.e., the collections $\mor(C_1,C_2)$ for $C_1,C_2\in\ob\cat C,$ are sets. Fix an object $P\in\ob\cat C.$

\begin{definition}{(The category of $P$-points.)} For an object $X\in\ob\cat C$ we define the category $\cat{Pts}_P(X),$ the category of $P$-points in $X,$ as the category with objects the set $$\ob\cat{Pts}_P(X)=\operatorname{pts}_P(X)=\mor(P,X),$$ and where a morphism from $x:P\rightarrow X$ to $y:P\rightarrow X$ is a commutative diagram $$\xymatrix{X\ar[rr]^\phi&&X\\&P.\ar[ul]^x\ar[ur]_y&}$$
\end{definition}
Notice that $\cat{Pts}_P(X)$ is a concrete category. Let $B\subseteq\ob\cat C$ be a sub-collection of the objects in $\cat C$ which objects we will call $B$-points, or base-points, if $B$ is evident. We extend the definition above.

\begin{definition}{(The category of $B$-points.)} For an object $X\in\ob\cat C$ we define the category $\cat{Pts}_B(X),$ the category of base-points in $X,$ as the category with objects the set $$\ob\cat{Pts}_B(X)=\operatorname{pts}_B(X)=\coprod_{P\in B}\mor(P,X),$$ and where a morphism from $x:P_1\rightarrow X$ to $y:P_2\rightarrow X$ is a commutative diagram $$\xymatrix{X\ar[r]^\phi&X\\P_1\ar[u]^x\ar[r]&P_2.\ar[u]_y}$$
\end{definition}

We notice that $\cat{Pts}_B(X)$ is a concrete category. Consider the category $\cat C$ with fixed base points $B\subseteq\ob\cat C.$ For each object $C$ we want to fix a unique base-point, if it exists, $c\in\pts_B(C).$ If we define $E(B)$ as the collections of $D\in\ob\cat C$ such that $\pts_B(D)=\emptyset,$ this is the function $f:\ob\cat C\setminus E(B)\rightarrow\underset{C\in\ob\cat C\setminus E(B)}\coprod\pts_B(C)$ giving each object $C$  a fixed base-point $f(C)\in\pts_B(C),$ when it exists.

\begin{definition}{(Subcategories of base-pointed objects.)}
Let $\overline{\cat C}_B$ denote the category with objects $\{(C,c)|C\in\ob\cat C,\pts_B(C)\neq\emptyset, c=f(C)\in\operatorname{pts}_B(C)\}$ and with morphisms $\phi:(C,c)\rightarrow (D,d)$ the commutative diagrams \begin{equation}\label{diag1}
\xymatrix{C\ar[r]^\phi&D\\P_1\ar[u]^c\ar[r]&P_2.\ar[u]_d}\end{equation} We will call any full subcategory $\cat C_B\subseteq\overline{\cat C}_B$ a base-pointed subcategory of $\cat C.$
\end{definition}

In a commutative diagram (\ref{diag1}) we will use the notation $\phi(c)=d.$
Given an object $X\in\ob\cat C$ and a base-point  $x_P:P\rightarrow X, P\in B.$ Let $\cat C_B$ be a base-pointed subcategory of $\cat C,$ and consider the contravariant functor $F:\cat C_B\rightarrow\sets$ defined by $F(Y,y_Q)=\{\rho:Y\rightarrow X|\rho(y_Q)=x_P\}.$ This functor is represented by $((X_x,x_R),\rho)$ if the morphism $\rho$ fitting in the below diagram  exists and is unique in the sense given in Definition \ref{repdef}.
$$\xymatrix{X_x\ar[r]^\rho&X\\R\ar[u]^{x_R}\ar[r]&P.\ar[u]_{x_P}}$$

\begin{definition}\label{repdef} Let $X\in\ob\cat C$ and $x_P:P\rightarrow X.$ Then the localization of $X$ in $x,$ if it exists, is defined as a base-pointed object $(X_x,x_R)\in\ob\cat C_P$ characterized by the following universal property: There is a morphism $\rho:X_x\rightarrow X$ such that $\rho(x_R)=x_P,$ and if  $(L,l_Q)$ is a base-pointed object with a morphism $\gamma:L\rightarrow X$ such that $\gamma(l_Q)=x_P,$ there exists a unique morphism $\kappa:L\rightarrow X_x$ such that $\kappa(l_Q)=x_R.$
\end{definition}

We notice that the existence of localization in a $P$-point is dependent on the choice of a base-pointed subcategory, usually called localizing subcategories.

\begin{example}\label{Groupexample} In the category $\cat{Grps}$ of groups we fix the free (abelian) group on one element $P=\langle x\rangle.$ For every group $G$ we have a bijection $G\simeq\mor(P,G)$ as every group-homomorphism $\phi:P\rightarrow G$ is determined by its value on $x.$ In this case we consider the  $P$-pointed subcategory with objects $$\ob\cat C_P=\{(\langle x\rangle/\langle x^p\rangle,x)|p\text{ prime}\}.$$
In this example, we let the base-points be the single object $P,$ that is $B=\{P\}.$ A morphism in this subcategory is a group homomorphism $f:\langle x\rangle/\langle x^q\rangle\rightarrow\langle x\rangle/\langle x^p\rangle,$ implying that $p=q.$ Thus for any group $G$ the localization $G_x$ in $x$ is the group $\langle x\rangle/\langle x^p\rangle$ where $p$ is the smallest prime such that $x^p=1.$ Notice the close connection to $p$-groups.
\end{example}

\begin{example} 
Let $k$ be a field and let $\cat{CAlg}_k$ denote the category of commutative $k$-algebras. As usual, we choose to consider the opposite category $\cat{CAlg}_k^o,$ 
and fix the object $P=k.$ Again, we let the base-points be the set $\{P\}$ with one element. As $P$-pointed subcategory we choose the category of all local $k$-algebras $A_\mathfrak m$ such that $A/\mathfrak m\simeq k.$
Then the localization of $A$ in $x:A\rightarrow k, \mathfrak m=\ker x,$ is a local algebra $A_{\mathfrak m}$ with maximal ideal $\mathfrak n,$ characterized by the following universal property: There exists a homomorphism $\rho:A\rightarrow A_{\mathfrak m}$ such that $\rho^{-1}(\mathfrak n)=\mathfrak m,$ and if there is another local algebra $B$ with maximal ideal $\mathfrak q$ and a homomorphism $\gamma:A\rightarrow B$ such that $\gamma^{-1}(\mathfrak q)=\mathfrak m$ then there is a unique morphism $\phi:A_{\mathfrak m}\rightarrow B$ with $\phi^{-1}(\mathfrak q)=\mathfrak n.$
\end{example}

\begin{example}
Let $k$ be a field and let $\cat{Vec}_k$ be the category of vector spaces over $k.$ As base-points in $\cat{Vec}_k$ we choose the single object $P=k,$ and as base-pointed subcategory, we choose the single object $(k,\id).$ Then the localization of a vector space $V$ in a $P$-point $x_{P}$ is isomorphic to the one-dimensional vector space $\operatorname{Span}(v_p), v_p=x_{P}(1).$ 
\end{example}

\begin{example}
The easiest example is maybe in the category of sets. Then we choose a set with one element $P=\{0\}$ as the only base-point, and the base-pointed subcategory as the sets with exactly one element. Then a localization of a set $S$ in $x_P$ is the one pointed set $\{x_P(0)\}.$
\end{example}

\begin{example}
Let $\cat{Daff}$ be the category of differentiable manifolds diffeomorphic to $\mathbb R^n$ for some $n\in\mathbb N.$ Morphisms in this category are $C^\infty$-functions. As base-points $B$ we choose the manifold $\mathbb R$ and as base-pointed subcategory the one-dimensional manifolds, i.e., diffeomorphic to $\mathbb R.$ The localization of an affine manifold $M$ in a point $p\in M$ is diffeomorphic to $\mathbb R.$   
\end{example}

The next example (and Example \ref{Groupexample}) proves the richer structure in algebra.

\begin{example} Let $\cat{CRing}$ be the category of commutative rings with unit. The base-points are formed by the collection of fields $k$, and we let the base-pointed subcategory be the subcategory of local rings $L$ with base point $\kappa:L\rightarrow L/\mathfrak m$ where $\kappa$ is the quotient map and $\mathfrak m$ is the unique maximal ideal in $L.$ For a ring $A$ with base-point $x:A\rightarrow k$ we have $\mathfrak p=\ker x$ and and so there is a homomorphism $\rho:A\rightarrow A_{\mathfrak p}$ with $\rho^{-1}(\mathfrak p A_{\mathfrak p})=\mathfrak p,$ and such that if $B$ is any other local ring with this property, there is a unique morphism $A_{\mathfrak p}\rightarrow B.$ This is the ordinary definition  localization of the ring $A$ in the prime ideal $\mathfrak p.$  
\end{example}

\subsection{Schemes of Objects}

Let $\cat C$ be a locally small category with a fixed collection of base-points $B\subseteq\ob\cat C$ and a fixed base-pointed subcategory $\cat C_B.$ Assume that for each object $X$ in $\cat C$ the localization $X_x$ in all base-points $x$ exists, and that the localizations belong to a category where direct products and coproducts exist. Finally, we will assume that images and coimages exists in $\cat C.$ This is the case for all abelian categories, in particular in the examples above.

\begin{definition}{Image and coimage.}
Given a morphism $f:X\rightarrow Y.$ Then the image of $f,$ if it exists, is a monomorphism $m:I\rightarrow Y$ satisfying the following universal property:
\begin{itemize}
\item[1)] There exists a morphism $e:X\rightarrow Y$ such that $f=m\circ e.$
\item[2)] For any object $I'$ with a morphism $e':X\rightarrow I'$ and a monomorphism $m':I'\rightarrow Y$ such that $f=m'\circ e',$ there exists a unique morphism $v:I\rightarrow I'$ such that $m=m'\circ v.$
\end{itemize}
The coimage of $f,$ if it exists, is an epimorphism $c:X\rightarrow C$ satisfying the following universal property:
\begin{itemize}
\item[1)] There exists a morphism $m:C\rightarrow Y$ such that $f=m\circ e.$
\item[2)] For any object $C'$ with a epimorphism $e':X\rightarrow C'$ and a morphism $m':C'\rightarrow Y$ such that $f=m'\circ e',$ there exists a unique morphism $v:I'\rightarrow I$ such that $m=v\circ e'=e.$
\end{itemize}
\end{definition}

In the following, we treat the covariant and the contravariant categories as two separate cases, starting with the covariant case.

For any object $X$ in $\cat C,$ let $O(X)=\underset{x\in B}\coprod X_x.$ By the universal property of coproducts, there is a unique morphism $$\gamma:O(X)\rightarrow X,$$ and we put $\mathcal O(X)=C(\gamma)$ where $C(\gamma)$ is the coimage of $\gamma.$ 

In the opposite category, put $O(X)=\underset{x\in B}\prod X_x.$ Then there is a unique morphism $$\gamma:X\rightarrow O(X),$$ and we put $\mathcal O(X)=I(\gamma)$ where $I(\gamma)$ is the image of $\gamma.$

From now on, we assume that $\cat C$ is a site, i.e., that it has a Grothendieck topology.

\begin{definition}\label{affschdef} We call $\mathcal O(X)$ the global object of $X$ over $B.$ If $\mathcal O(X)\simeq X,$ we call $X$ an affine object in $\cat C.$ If $X$ has a covering $\{U_i\rightarrow X\}_{i\in I}$ such that $\mathcal O(U_i)\simeq U_i,$ that is, $X$ has an affine covering, we call $X$ an affine scheme of objects in $\mathcal C.$ If all objects in $\cat C$ are affine schemes, we say that $\cat C$ is complete. 
\end{definition}

Let $X=\{U_i\}_{i\in I}$ be an indexed collection of objects in the complete site $\cat C.$ Define $\mathcal O_X=\underset{\underset{i\in I}\leftarrow}\lim\ \mathcal O(U_i)$ where the limit is taken over the coverings. We have that $\pts_B(\mathcal O_X)=X,$ and that if $Y=\{U_j\}_{j\in J}$ is another collection of objects, then a morphism $f^\ast:\mathcal O_Y\rightarrow\mathcal O_X$ induces a morphism $f:\pts_B(X)\rightarrow\pts_B(Y).$

\begin{definition}\label{Schdef} Let $\cat C$ be complete site. A $\cat C$-scheme is a collection $X=\{U_i\}_{i\in I}$ of indexed objects in $\cat C.$ A morphism of $\cat C$-schemes, is a map of sets $f:X\rightarrow Y$ induced by a $\cat C$-morphism $f^\ast:\mathcal O_Y\rightarrow\mathcal O_X.$ Notice that this defines a site $\cat{\cat C}-\cat{Sch}$ and that that $\mathcal O(X)=\mathcal O_X$ defines a sheaf on this site.
\end{definition}

\end{document}

%% file: defn.tex
\DeclareMathOperator{\id}{id}

\DeclareMathOperator{\im}{im}

\DeclareMathOperator{\mor}{Mor}

\DeclareMathOperator{\spec}{Spec}

\DeclareMathOperator{\ob}{ob}

\DeclareMathOperator{\pts}{Pts}

\newcommand{\op}{\mathit{op}}

\newcommand{\cat}[1]{\mathbf{#1}}

\newcommand{\sets}{\mathbf{Sets}}

\newcommand{\grmcat}[1]

